\documentclass[11pt,reqno]{amsart}
\date{Agust 29, 2011}

\usepackage{amsmath,amsfonts,amsthm,amssymb,amsxtra}
\usepackage{amsxtra, amssymb, mathrsfs, pstricks}

\newtheorem{theorem}{Theorem}[section]
\newtheorem{proposition}[theorem]{Proposition}

\newtheorem{lemma}[theorem]{Lemma}
\newtheorem{corollary}[theorem]{Corollary}

\theoremstyle{definition}

\newtheorem{definition}[theorem]{Definition}

\theoremstyle{remark}

\newtheorem{remark}[theorem]{\bf Remark}

\numberwithin{equation}{section}

\renewcommand{\epsilon}{\varepsilon}

\newcommand{\N}{\mathbb{N}}

\newcommand{\R}{\mathbb{R}}

\DeclareMathOperator{\gen}{gen}

\DeclareMathOperator{\tr}{tr}

\begin{document}

\title[Heat kernels of metric trees]{Heat kernels of metric trees and applications}

\author{Rupert L. Frank}
\address{Rupert L. Frank, Department of Mathematics,
Princeton University, Washington Road, Princeton, NJ 08544, USA}
\email{rlfrank@math.princeton.edu} 

\author {Hynek Kova\v{r}\'{\i}k}
\address{Hynek Kova\v{r}\'{\i}k, Department of Mathematics,
Politecnico di Torino, Corso Duca degli Abruzzi 24, 10129 Torino, ITALY}
\email{hynek.kovarik@polito.it}

\thanks{ Support through DFG grant FR 2664/1-1 (R.F.) and U.S. NSF
grants PHY-0652854 (R.F.) is gratefully acknowledged. H.K. was partially supported by the MIUR-PRIN08 grant for the project  ''Trasporto ottimo di massa, disuguaglianze geometriche e funzionali e applicazioni''}

\begin{abstract}
We consider the heat semigroup generated by the Laplace operator on metric trees. Among our results we show how the behavior of the associated heat kernel depends on the geometry of the tree. As applications  we establish new eigenvalue estimates for Schr\"odinger operators on metric trees.  
\end{abstract}

\keywords{Heat kernel, metric tree, eigenvalue estimate, Sobolev inequality}

\maketitle


\section{Introduction}

The relation between functional inequalities, like Sobolev inequalities for functions on (smooth or singular) manifolds, and geometric properties of this manifold, like volume growth, has been studied extensively. The crucial link between these two fields is the heat kernel of the Laplace operator on the manifold. Indeed, many functional inequalities have an equivalent form as heat kernel bounds; see, e.g., the book \cite{Da}, the surveys \cite{cou,grig,sc} and the references therein. What is less known, is that there is a close relation between heat kernel bounds and so-called Lieb--Thirring and Cwikel--Lieb--Rozenblum inequalities. The latter are inequalities about eigenvalues of Schr\"odinger operators and are closely related to Sobolev-like inequalities for systems of functions. These inequalities go back to \cite{LT} and their relation to heat kernel bounds are discussed, for instance, in \cite{LS,fls}.

Our goal in this paper is to explore the abstract connection between heat kernel bounds and Lieb--Thirring inequalities in the concrete context of metric graphs. In particular, we shall derive precise bounds on the heat kernel which reflect the geometric properties of the graph. A metric graph is a combinatorial graph where the edges are considered as one-dimensional intervals. The fact that the edges are represented by non-degenerate line segments allows us to introduce various differential operators on metric graphs, such as Laplace and Schr\"odinger operators. Motivated by the fact that these operators on metric graphs appear in a number of models from mathematical physics (where they are often called \emph{quantum graphs}), their analysis has recently attracted a lot of interest; see, e.g., the proceedings \cite{EKKST,BCFK} and the references therein. Various functional inequalities for Laplace operators on metric graphs were recently studied in \cite{ehp,NS1,F,efk1}. The papers \cite{NS2,k,efk,dh,sol} contain inequalities about eigenvalues of Schr\"odinger operators on metric graphs. For a recent study of heat kernels on combinatorial graphs we refer to \cite{klw}.

Most (but not all) of our results are valid for a special class of metric graphs, namely, \emph{sparse, symmetric metric trees}. A metric \emph{tree} is a metric graph on which any two points can be connected by a unique path. We will also suppose that the tree has a {\it root}. By \emph{symmetric} we mean that the length of each edge and the branching number of each vertex depend only on the distance to the root. Finally, under \emph{sparse} we understand an infinite tree with unbounded edge lengths. Later on we will quantify this growth condition and introduce the notion of a global dimension.

We introduce the Neumann Laplacian $-\Delta_\Gamma^N$ on a rooted metric tree as the self-adjoint operator in $L_2(\Gamma)$ which acts as the usual one-dimensional Laplacian on the edges and satisfies the Kirchhoff matching conditions at the vertices and a Neumann boundary condition at the root. (Most of our results extend to the easier case of a Dirichlet condition at the root.) The heat kernel is the integral kernel of the operator $e^{t \Delta_\Gamma^N}$, $t >0$, i.e.,
\begin{equation} \label{int-kernel}
k(x,y,t) := e^{t \Delta_\Gamma^N}(x,y), \qquad x,y \in\Gamma \,,
\end{equation}
and, as we explained, we are going to study its dependence on the time $t$ and the volume growth of the metric tree. The latter is encoded in the so-called branching function 
\begin{equation} \label{g0-first}
g_0(r) = \# \{x\in \Gamma\, :\, |x| =r \}, \quad r \geq 0,
\end{equation}
where $|x|$ denotes the distance between $x$ and the root of $\Gamma$.

The particular feature of this set-up is that $\Gamma$ represents a structure of \emph{mixed dimensionality}. On one hand, it is locally one-dimensional and therefore one should expect a $t^{-1/2}$ singularity of $k(x,x,t)$ for small times. In fact, we show that (Theorem~\ref{1-dim}), up to a multiplicative constant, the power function $t^{-1/2}$ serves as a uniform upper bound on $k(x,y,t)$. This result does not use that $\Gamma$ is a tree.

On the other hand, since the global dimensionality of $\Gamma$ depends on the behavior of $g_0$ at infinity, it is natural to expect that the large time decay of the heat kernel \eqref{int-kernel}, for fixed $x$ and $y$, will be determined by the growth of $g_0$. In order to be able to quantify the decay rate precisely, we assume that $\Gamma$ is a symmetric tree. One of our main results (Theorem~\ref{hk-estimate}) says that, if $g_0$ does not grow too fast, then  
\begin{equation} \label{eq-1}
k(x,x,t) \, \leq\,  \frac{C\, g_0(|x|)}{\sqrt{t}\, \, g_0(|x|+\sqrt{t})}
\end{equation}
for all $x\in\Gamma,\, t\geq 0$ and some constant $C>0$. Moreover, we show that this bound is order-sharp in the decay rate with respect to $t$ for fixed $x\in\Gamma$. Inequality \eqref{eq-1} is in the spirit of the works of Grigor'yan and Saloff-Coste on Riemannian manifolds and our proof relies on their approach. In the special case where the branching function grows like a constant times $r^{d-1}$ for some $d\geq 1$ (i.e., $\Gamma$ has global dimension $d$, see Definition \ref{dim-rec}) we obtain an upper bound on $k(x,x,t)$ by a constant times $g_0(|x|) t^{-d/2}$. We also obtain bounds in certain cases where $g_0$ grows faster than polynomially, e.g., exponentially as in the case of a homogeneous tree.

As an application of these heat kernel bounds we will derive inequalities for the negative eigenvalues of Schr\"odinger operators 
$$
-\Delta_\Gamma^N - V \qquad \text{in} \quad L_2(\Gamma)
$$
with an electric potential $-V$ decaying at infinity. Using our new heat kernel bounds we are able to extend our previous results from \cite{efk} in two ways. First, we are able to remove the assumption that $V$ has to be radially symmetric. Secondly, while in \cite{efk} we treated the different dimensionality by proving a \emph{one-parameter family} of inequalities, we now establish a stronger \emph{single inequality} which takes separately account of the regions where $V$ is large (in this case only the local dimension counts) and where $V$ is small (in this case only the global dimension counts). As a sample of our result we assume again that $\Gamma$ is a symmetric metric tree whose branching function grows like a constant times $r^{d-1}$ for some $d> 1$. Then for all $\gamma\geq 1/2$ one has
\begin{align*}
\tr\left(-\Delta^N_\Gamma -V\right)_-^\gamma \, \leq \, L_d(\beta,\gamma)
& \int_{g_0(|x|)^{\frac2{d-1}}V(x)_+<\beta} V(x)_+^{\gamma+\frac{d}{2}}\, g_0(|x|)\, dx  \\
+ \tilde L_d(\beta, \gamma) & \int_{g_0(|x|)^{\frac2{d-1}}V(x)_+\geq \beta} V(x)_+ \, g_0(|x|)^{\frac{1-2\gamma}{d-1}}\, dx \,,
\end{align*}
where $\beta>0$ is an arbitrary parameter. We obtain similar inequalities also for a certain range of $\gamma<1/2$ and, in particular, for $\gamma=0$, that is, for the number of negative eigenvalues, provided $d>2$. It is remarkable that Lieb--Thirring and Cwikel--Lieb--Rozenblum inequalities hold on metric trees in a form not too different from their Euclidean form despite the fact that the spectrum of the Laplacian $-\Delta^N_\Gamma$ is purely singular \cite{BF}.


\section{Main results}

\subsection{Preliminaries} \label{prelim} 

We denote by $\mathcal{E}$ the set of edges and by $\mathcal{V}$ the set of vertices of a rooted tree graph $\Gamma$. For a vertex $v\in \mathcal{V}$ we define its generation $gen(v)$ as the number of vertices (including the starting point but excluding the end point) which lie on the unique path connecting $v$ with the root $o$.  We call  $\Gamma$  {\it symmetric} if all the vertices at the
same distance from the root have equal branching numbers and all the
edges emanating from these vertices have equal length.

In this case for a vertex of generation $l$ we denote by $b_l$ its (forward) branching number and by $r_l$ its distance to the root. We also set $b_0=1$ and $r_0=0$, and assume that $b_l \geq 2$ for any  $l \geq 1$. We emphasize that this assumption implies that $\Gamma$ has no other leaves than $o$.
We introduce the branching function
\begin{align} \label{g0}
g_0(r) := b_0\, b_1\cdots b_l \qquad \text{if} \quad r_l < r \leq
r_{l+1}, \quad l\in\N_0.
\end{align}
Throughout we assume that 
$$
 \sup_{x\in\Gamma} |x| = \infty, \qquad |x|:= \text{dist}(x,o). 
$$

\noindent The Neumann Laplacian $-\Delta_\Gamma^N$ is defined as the non-negative self-adjoint operator in $L_2(\Gamma)$ associated with the closed quadratic form
\begin{align}\label{eq:kinetic}
  \int_{\Gamma} |\varphi'(x)|^2 \, dx, \quad \varphi \in H^1(\Gamma).
\end{align}
Here $H^1(\Gamma)$ consists of all continuous
functions $\varphi$ such that $\varphi \in H^1(e)$ on each edge $e\in\mathcal{E}$ and
\begin{align*}
\int_{\Gamma} \left(|\varphi'(x)|^2 + |\varphi(x)|^2\right) \, dx <
\infty.
\end{align*}
From \eqref{eq:kinetic} and  the Beurling-Deny  theorem (see, e.g., \cite[Sec. XIII.12]{ReSi4}) it follows that for any $t>0$, the operator $e^{t \Delta_\Gamma^N}$ is a positivity preserving contraction on $L_p(\Gamma)$ for every $1\leq p\leq \infty$. We denote the integral kernel of this operator by $k(\cdot,\cdot,t)$, that is,
$$
\left( e^{t \Delta_\Gamma^N} f\right)(x) = \int_\Gamma k(x,y,t) f(y) \,dy \,.
$$
 
\smallskip

\noindent A special class of symmetric trees are those for which the branching function $g_0$ has a power like growth at infinity. For such trees, following \cite{k},  we define their global dimension:

\begin{definition} \label{dim-rec}
A symmetric tree $\Gamma$ has global dimension $d\geq 1$ if 
\begin{align}\label{eq:dim}
0 <  \  \inf_{r\geq 0} \frac{g_0(r)}{(1+r)^{d-1}} \ \leq \ 
\sup_{r\geq0} \frac{g_0(r)}{(1+r)^{d-1}}  \ <\infty\, .
\end{align}
\end{definition}

\smallskip


\subsection{Main results}

 Our first result gives an upper bound on the heat kernel not only for trees, but for a very general class of graphs. 

\begin{theorem} \label{1-dim}
 Let $\Gamma$ be a connected graph of infinite volume. Then for all $x\in\Gamma$ and all $t>0$ it holds
$$
k(x,x,t) \leq (\pi\, t)^{-1/2} \,.
$$
\end{theorem}

The constant in this estimate is best possible, as can be easily seen from the example $\Gamma=\R_+$ where $k(x,x,t)=(4\pi t)^{-1/2} (1+ e^{-x^2/t})$.

Theorem \ref{1-dim} gives an upper bound on $k(x,x,t)$  which is uniform with respect to $x$. However, as explained above, for a fixed $x\in\Gamma$ the large time decay of $k(x,x,t)$ should depend on the growth of $g_0$. This phenomenon is reflected in our next result.

\begin{theorem} \label{hk-estimate}
Let $\Gamma$ be a symmetric tree. Assume that the branching function $g_0$ of $\Gamma$ satisfies  
\begin{equation} \label{xdouble}
g_0(2r) \, \leq\, C_0\, g_0(r) \qquad \text{for all}\ r \in[0,\infty)
\end{equation}
for some constant $C_0$. Then there is a constant $c>0$ such that for all $x\in\Gamma$ and all $t>0$ it holds
\begin{equation} \label{eq:two-sided}
\frac{1}{c\, \sqrt{t}\, \, g_0(|x|+\sqrt{t})}\  \leq\  k(x,x,t) \ \leq \ \frac{c\, g_0(|x|)}{\sqrt{t}\, \, g_0(|x|+\sqrt{t})}\, .
\end{equation}
\end{theorem}

One can obtain less precise bounds than \eqref{eq:two-sided} by using the monotonicity of $g_0$. Bounding $g_0(|x|)\leq g_0(|x|+\sqrt t)$ we see that \eqref{eq:two-sided} yields the one-dimensional estimate $k(x,x,t) \leq c\, t^{-1/2}$ from Theorem \ref{1-dim}. However, Theorem \ref{1-dim} is valid for a larger class of graphs and gives an explicit (and sharp) value of the constant $c$.

A different use of the monotonicity, namely $g_0(|x|+\sqrt{t})^{-1} \leq g_0(\sqrt{t})^{-1}$, yields the bound
$$
k(x,x,t) \ \leq \ \frac{c\, g_0(|x|)}{\sqrt{t}\, \, g_0(\sqrt{t})}\, .
$$
One can control the term $g_0(\sqrt t)$ if one assumes that $\Gamma$ has a global dimension in the sense of Definition \ref{dim-rec}. Noting that this assumption implies condition \eqref{xdouble}, we obtain

\begin{corollary} \label{finite-dim}
Let $\Gamma$ be a symmetric tree with global dimension $d$. Then
for some $C>0$, any $x\in\Gamma$ and any $t>0$ it holds
\begin{equation} \label{d:dim}
k(x,x,t) \, \leq \, C \ t^{-d/2}\,  g_0(|x|).
\end{equation}
\end{corollary}

In view of the lower bound in \eqref{eq:two-sided} the decay rate $t^{-d/2}$ for fixed $x\in\Gamma$ in the above upper bound is best possible.

To summarize our analysis so far, we have seen that the behavior of the heat kernel for small and large times is determined by the local and global dimension of $\Gamma$, respectively. Indeed, for a tree of global dimension $d$ Theorem \ref{hk-estimate} gives
\begin{equation} \label{local-global}
k(x,x,t) \simeq t^{-1/2} \quad \text{as\ } t\to 0, \qquad k(x,x,t) \simeq t^{-d/2} \quad \text{as\ } t\to\infty.
\end{equation}

\smallskip 

 There are of course symmetric trees which do not satisfy condition \eqref{xdouble}. This happens typically when the branching function $g_0$ grows too fast. To cover such cases we have 

\begin{theorem} \label{no-vd}
Let $\Gamma$ be a symmetric tree and assume that 
\begin{equation}\label{eq:sgamma}
S_\Gamma^{-1}(\delta) := \sup_{r>0} \left( \int_0^r g_0(s)\,ds
\right)^{(\delta-2)/\delta} \left( \int_r^\infty \frac{ds}{g_0(s)} \right)  < \infty
\end{equation}
for some $\delta >2$. Then for any $x\in\Gamma$ and $t>0$ it holds
\begin{equation} \label{eq:no-vd}
k(x,x,t)   \, \leq \, \Big(\frac
\delta{2 \tilde S_\Gamma(\delta)}\Big)^{\delta/2}\  t^{-\delta/2}\, g_0(|x|),
\end{equation}
where
\begin{equation}
\label{eq:sgammatilde}
\tilde S_\Gamma(\delta)  := \left( \frac{(\delta-2)^{\delta-2} \delta^\delta}{(2(\delta-1))^{2(\delta-1)}} \right)^{1/\delta} S_\Gamma(\delta).
\end{equation}
\end{theorem}

\smallskip

It is easily seen that if $\Gamma$ has a global dimension $d>2$, then Theorem \ref{no-vd} is applicable with $\delta=d$ and \eqref{eq:no-vd} recovers the bound from Corollary \ref{finite-dim} (with an explicit constant).

 Contrary to Theorem \ref{hk-estimate}, estimate \eqref{eq:no-vd} is applicable also when $g_0(r) \simeq \exp(C r^\alpha)$ for some $\alpha >0$. In this case Theorem \ref{no-vd}  says that $k(x,x,t) = \mathcal{O}(t^{-n})$ as $t\to\infty$ for any $n\in\N$. Moreover, it is not difficult to check  that for $0< \alpha< 1$ we have $\inf\, $spec$(-\Delta_\Gamma^N) =0$ and therefore the heat kernel decay should be sub-exponential for such values of $\alpha$. 

On the other hand, for $\alpha=1$ it turns out that $\inf\, $spec$(-\Delta_\Gamma^N) >0$, which shows that $k(x,y,t)$ must decay exponentially fast in this case. 
As an example of such trees we will consider so-called \emph{homogeneous} trees. A tree $\Gamma_b$ is called homogeneous if all the edges have the same length $\tau$ and if the branching number $b_k=b\geq2$ is independent of $k$. By scaling we may assume that $\tau=1$. The branching function $g_0=:g_b$ then reads
\begin{equation*}
g_b(r) = b^j, \qquad j< r \leq j+1, \quad j\in\N_0\, .
\end{equation*}
The corresponding Laplacian $-\Delta_{\Gamma_b}^N$ is positive definite and 
\begin{equation} \label{lambdab}
\lambda_b := \inf \text{spec} \big(-\Delta_{\Gamma_b}^N\big) = \Big(\arccos \frac 1{R_b} \Big)^2, \qquad R_b =
\frac{b^{\frac 12}+b^{-\frac 12}}{2}\, ,
\end{equation}
see \cite{SS}.  

\begin{theorem} \label{infinite-dim}
Let $\Gamma_b$ be a homogeneous tree with edge length $1$ and branching
number $b\geq 2$. Then there is a constant $C_b$ such that for any $x\in\Gamma_b$ and any $t>0$ it holds
\begin{equation} \label{eq:homog-hk}
k(x,x,t)\ \leq \  C_b \ e^{-\lambda_b t}\,t^{-3/2}\,  (1+|x|)^2 \,.
\end{equation}
\end{theorem}

We conclude this section with a brief outline of the rest of the paper. In Section \ref{sec:1-dim} we prove the general bound from Theorem \ref{1-dim}. The proof of Theorems \ref{hk-estimate}, \ref{no-vd} and \ref{infinite-dim} is based on the fact \cite{Ca,NS1} that the Laplace operator on a symmetric tree can be decomposed into a direct sum of weighted Laplace operators on half-lines. In Section \ref{sect:identity} we show that the diagonal element $k(x,x,t)$ can be bounded from above in terms of the heat kernel of the weighted manifold $([0,\infty), g_0 dr)$, see Theorem \ref{main-estim}. In Section \ref{sect:proofs} we then establish a suitable estimates on the heat kernel of $([0,\infty), g_0 dr)$, see Proposition \ref{basic-estim}.

As we have explained in the introduction, one important consequence of the heat kernel bounds from Theorems \ref{1-dim}, \ref{hk-estimate}, \ref{no-vd} and \ref{infinite-dim} are eigenvalue estimates for Schr\"odinger operators on metric trees. In Section \ref{sect:appl} we will state these inequalities and we will see how the heat kernel method from this section allows us to fundamentally improve upon our previous results in \cite{efk}.


\section{Proof of Theorem \ref{1-dim}}\label{sec:1-dim}

We begin by proving the general heat kernel bound from Theorem \ref{1-dim} for arbitrary graphs. The key ingredient is the following logarithmic Sobolev inequality.

\begin{proposition}[Logarithmic Sobolev inequality] \label{log}
Let $\Gamma$ be a connected graph of infinite volume. Then for any
$u\in H^1(\Gamma)$ and any $a>0$
$$
\frac{a^2}{\pi} \int_\Gamma |u'|^2 \,dx \geq \int_\Gamma |u|^2
\ln\left(\frac{|u|^2}{\|u\|^2}\right)
 \,dx + (1+\ln (a/2)) \int_\Gamma |u|^2 \,dx \,.
$$
\end{proposition}

\begin{proof}
For any function $u$ on $\Gamma$ vanishing at infinity we consider
its rearrangement $u^*$. By definition, this is the unique
non-increasing function on $[0,\infty)$ with the same distribution
function as $|u|$. In particular, this property implies that
$$
\int_\Gamma |u|^2 \,dx = \int_0^\infty |u^*|^2 \,dr
\quad\text{and}\quad \int_\Gamma |u|^2 \ln |u|^2 \,dx =
\int_0^\infty |u^*|^2 \ln |u^*|^2 \,dr \,.
$$
Moreover, the same argument as in \cite{F} yields that
$$
\int_\Gamma |u'|^2 \,dx \geq \int_0^\infty |(u^*)'|^2 \,dr \,.
$$
Therefore, the assertion follows from the Euclidean logarithmic
Sobolev inequality (applied to symmetric functions), see, e.g., \cite[Thm. 8.14]{LL}.
\end{proof}

\begin{proof}[Proof of Theorem \ref{1-dim}]
By Proposition \ref{log} and the Carlen--Loss argument as presented in \cite[Thm. 8.18]{LL} it follows that
$$
\| e^{t \Delta_\Gamma^N} \|_{L_1(\Gamma)\to L_\infty(\Gamma)} =  \sup_{x,y\in\Gamma} k(x,y,t) \, \leq \, (\pi \, t)^{-1/2}
$$
as required.
\end{proof}



\section{Decomposition of the heat kernel on symmetric trees}
\label{sect:identity}

 In this section we make use of the symmetry of $\Gamma$ and we shall see that it allows us to decompose the heat kernel of $-\Delta_\Gamma^N$ into a sum of heat kernels of one-dimensional operators. This is, of course, reminiscent of the partial wave decomposition in Euclidean space. In the context of metric trees this decomposition is due to \cite{Ca,NS1}.

We recall the definition of the branching function $g_0$ from \eqref{g0} and we introduce the higher order branching functions $g_l$ for $l\in\N$ by
\begin{align} \label{gk}
g_l(r) := \left\{
\begin{array}{l@{\qquad}l}
 0, & r < r_l\, , \\
 1, & r_l \leq r <  r_{l+1}\, , \\
 b_{l+1}b_{l+2}\cdots b_n, & r_n \leq r< r_{n+1},\quad n>l \, .
\end{array}
\right.
\end{align}


\subsection{Partial wave decomposition}

Let $v$ be a vertex of generation $l\in\N$ and denote by $\Gamma_{v,m}$, $m=1,\ldots,b_l$, the mutually disjoint (forward) subtrees rooted at $v$ and by $\chi_{v,m}$ the corresponding characteristic functions. We shall also use the notation $\Gamma_v:=\bigcup_{m=1}^{b_l} \Gamma_{v,m}$. Moreover, let
$$
\omega_l:= e^{2\pi i/b_l}
$$
and put
\begin{equation} \label{y}
Y_{l,v,\sigma}(x) := \frac{1}{\sqrt{b_l\ g_l(|x|)}} \sum_{m=1}^{b_l}
\omega_l^{m \sigma} \chi_{v,m}(x)
\end{equation}
for $\sigma=1,\ldots,b_l-1$ if $l\geq 1$ and
$$
Y_{0,o,1}(x) := \frac{1}{\sqrt{g_0(|x|)}} \,.
$$
If $f$ is any function on $\Gamma$ we put
$$
f_{l,v,\sigma}(r) := \frac{1}{\sqrt{g_l(r)}} \sum_{|x|=r}
\overline{Y_{l,v,\sigma}(x)} \ f(x) \,.
$$
Below we denote by $\sum_{l,v,\sigma}$ summation over all
$l\in\N_0$, all $v$ with $\gen v=l$ and all $\sigma$ with
$1\leq\sigma\leq b_l-1$ (respectively, only $\sigma=1$ if $l=0$).

\begin{proposition}\label{decompf}
For any function $f$ on $\Gamma$ one has
\begin{equation}\label{eq:decompf}
 f(x) = \sum_{l,v,\sigma} f_{l,v,\sigma}(|x|) \ \sqrt{g_l(|x|)} \ Y_{l,v,\sigma}(x)
\end{equation}
and for a.e. $r> 0$
\begin{equation}
 \label{eq:decompnorm}
\sum_{|x|=r} |f(x)|^2 = \sum_{l,v,\sigma} |f_{l,v,\sigma}(r)|^2 \ g_l(r) \,.
\end{equation}
\end{proposition}

\begin{proof}
Both equalities \eqref{eq:decompf} and \eqref{eq:decompnorm} will
follow if we can show that
\begin{equation}\label{eq:decompbasic}
\sum_{l,v,\sigma} Y_{l,v,\sigma}(x)\  \overline{Y_{l,v,\sigma}(x')} = \delta_{x,x'}
\quad \text{if}\quad |x|=|x'| \,.
\end{equation}
Let $|x|=|x'|=r$ and define
\begin{equation}
N = \max \left\{\text{gen}(v)\, :\, x,x'\in \Gamma_v\right\}.
\end{equation}
Since \eqref{eq:decompbasic} is obvious for $r< r_1$, we may assume that $N\geq 1$. Then
$$
\sum_{l,v,\sigma} Y_{l,v,\sigma}(x)\  \overline{Y_{l,v,\sigma}(x')} =
\frac{1}{g_0(r)} + \sum_{l=1}^N\sum_{\text{gen}(v)=l}\sum_{\sigma=1}^{b_l-1}\,
Y_{l,v,\sigma}(x) \ \overline{Y_{l,v,\sigma}(x')}.
$$
(The first term on the right side is the contribution from $l=0$.) 
For each $1\leq l\leq N-1$ there is a unique vertex $v$ of the $l$th
generation and a unique $1\leq m\leq b_l$ such that $x,x'\in \Gamma_{v,m}$. Hence from \eqref{gk}, \eqref{g0} and \eqref{y} we get
\begin{align*}
\frac{1}{g_0(r)} + \sum_{l=1}^{N-1}\sum_{\text{gen}(v)=l}\sum_{\sigma=1}^{b_l-1}\,
Y_{l,v,\sigma}(x) \ \overline{Y_{l,v,\sigma}(x')}
& = \frac{1}{g_0(r)} + \sum_{l=1}^{N-1}\, \sum_{\sigma=1}^{b_l-1} \, \frac{1}{b_l g_l(r)} \\
& = \frac{1}{g_0(r)} + \frac{1}{g_0(r)}\, \sum_{l=1}^{N-1}\, b_0\cdots b_{l-1}(b_l-1) \\
& = \frac{b_0\cdots b_{N-1}}{g_0(r)} \,.
\end{align*}
In order to study the contribution from $l=N$ we suppose first that $x\neq x'$.
Then there is a unique vertex $w$ of the $N$th
generation and two numbers $m\neq m'$ such that $x\in\Gamma_{w,m}$ and
$x'\in\Gamma_{w,m'}$. Using the fact that
$$
\sum_{\sigma=1}^{b_N}\, e^{\frac{2\pi i \sigma (m-m')}{b_N}} = 0
$$
we find
\begin{align*}
\sum_{\text{gen}(v)=N}\sum_{\sigma=1}^{b_N-1}\,
Y_{l,v,\sigma}(x) \ \overline{Y_{l,v,\sigma}(x')}
= \frac{1}{b_N\, g_N(r)} \, \sum_{\sigma=1}^{b_N-1}\,  e^{\frac{2\pi
i \sigma (m-m')}{b_N}}
= - \frac{1}{b_N\, g_N(r)} \,.
\end{align*}
Assertion \eqref{eq:decompbasic} is now a consequence of
$$
\frac{b_0\cdots b_{N-1}}{g_0(r)} = \frac{1}{b_N\, g_N(r)} \,,
$$
see \eqref{gk} and \eqref{g0}.

Assume now that $x=x'$. In this case we have $g_N(r)=1$ and hence
$$
\sum_{\text{gen}(v)=N}\sum_{\sigma=1}^{b_N-1}\,
Y_{l,v,\sigma}(x) \ \overline{Y_{l,v,\sigma}(x')}
= \frac{b_N-1}{b_N} \,.
$$
Assertion \eqref{eq:decompbasic} in this case is a consequence of
$$
\frac{b_0\cdots b_{N-1}}{g_0(r)} + \frac{b_N-1}{b_N} = 1 \,,
$$
see again \eqref{gk} and \eqref{g0}. This concludes the proof of \eqref{eq:decompbasic}.
\end{proof}

\noindent By a duality argument \eqref{eq:decompbasic} is equivalent
to
$$
\sum_{|x|=r} Y_{l,v,\sigma}(x)\  \overline{Y_{l',v',\sigma'}(x)} =
\delta_{l,l'}\, \delta_{v,v'}\, \delta_{\sigma,\sigma'}.
$$

The reason for including the factor $1/\sqrt{g_l(r)}$ in the
definition of $f_{l,v,\sigma}$ is that in this way $f_{l,v,\sigma}$
inherits continuity and differentiability properties of $f$. More
precisely, one has

\begin{proposition}\label{decompf'}
If $u$ is a locally absolutely continuous function on $\Gamma$, then
the $u_{l,v,\sigma}$ are locally absolutely continuous. One has
\begin{equation}\label{eq:decompf'}
 u'(x) = \sum_{l,v,\sigma} u_{l,v,\sigma}'(|x|) \ \sqrt{g_l(|x|)} \ Y_{l,v,\sigma}(x)
\end{equation}
and for a.e. $r> 0$
\begin{equation}
 \label{eq:decompnorm'}
\sum_{|x|=r} |u'(x)|^2 = \sum_{l,v,\sigma} |u_{l,v,\sigma}'(r)|^2 \ g_l(r) \,.
\end{equation}
\end{proposition}

The first part of the assertion is easily verified, see also
\cite{NS1}. The second part is a consequence of Proposition
\ref{decompf}.


\subsection{The reduced operators and their heat kernels}

Let $A_l, \, l\in\N_0$, be the non-negative operators
in $L_2((r_l,\infty),g_l dr)$ generated by the quadratic forms
\begin{equation} \label{q-form}
\int_{r_l}^\infty |u'|^2 \,g_l \,dr
\end{equation}
with form domains
\begin{equation*}
H^1((0,\infty), g_0\, dr)\ \mathrm{if}\ l=0
\qquad\mathrm{and}\qquad
H^1_0((r_l,\infty),g_l\, dr) \ \mathrm{if}\ l\in\N \,.
\end{equation*}
Since, by Propositions \ref{decompf} and \ref{decompf'},
$$
\|u\|^2 = \sum_{l,v,\sigma} \int_{r_l}^\infty |u_{l,v,\sigma}|^2 \,g_l \,dr
\qquad\mathrm{and}\qquad
\|u'\|^2 = \sum_{l,v,\sigma} \int_{r_l}^\infty |u_{l,v,\sigma}'|^2 \,g_l \,dr \,,
$$
we have
$$
(-\Delta_\Gamma^N u)_{l,v,\sigma} = A_l\,  u_{l,v,\sigma} \,.
$$
Hence by the spectral theorem
\begin{equation}\label{eq:heatdecomp}
(e^{t\Delta_\Gamma^N}\, u)_{l,v,\sigma} = e^{-t A_l}\,
u_{l,v,\sigma} \,.
\end{equation}
Let $k(x,y,t)$  be the heat kernel of $-\Delta_\Gamma^N$ and let
$k_l(r,s,t), \, l\in\N_0$, be the heat kernels of the operators $A_l$, that is,
$$
(e^{t\Delta_\Gamma^N} u)(x) = \int_\Gamma k(x,y,t) u(y) \,dy ,
$$
and
\begin{equation} \label{Al}
(e^{-tA_l}f)(r) = \int_{r_l}^\infty k_l(r,s,t) f(s)\,g_l(s)\,ds \,.
\end{equation}
From the Beurling-Deny conditions (see, e.g., \cite[Sec. XIII.12]{ReSi4}) one infers that $k_l(r,s,t)$ is non-negative.
For technical reasons we extend $k_l(r,s,t)$ by zero to all negative
values of $(r-r_l)$ and $(s-r_l)$. Combining \eqref{eq:heatdecomp}
with \eqref{Al} we find
\begin{align*}
(e^{t\Delta_\Gamma^N} u)(x) & = \sum_{l,v,\sigma}\,
( \exp(t\Delta_\Gamma^N)\, u)_{l,v,\sigma} (|x|)\, \sqrt{g_l(|x|)}\,
\, Y_{l,v,\sigma}(x) \\
& =  \sum_{l,v,\sigma}\,  \int_{r_l}^\infty k_l(|x|,s,t)\,
\sqrt{g_l(s)}\, \, \sum_{|y|=s}\,  \sqrt{g_l(|x|)}\, \,
Y_{l,v,\sigma}(x)\, \overline{Y_{l,v,\sigma}(y)}\, \, u(y)\, ds \\
& = \int_0^\infty \sum_{|y|=s}\, \sum_{l,v,\sigma} k_l(|x|,|y|,t)\,
\sqrt{g_l(|x|)\ g_l(|y|)}\, \, Y_{l,v,\sigma}(x)\,
\overline{Y_{l,v,\sigma}(y)}\, u(y)\, ds\\
& = \int_\Gamma \sum_{l,v,\sigma} k_l(|x|,|y|,t)\, \sqrt{g_l(|x|)\
g_l(|y|)}\, \, Y_{l,v,\sigma}(x)\, \overline{Y_{l,v,\sigma}(y)}\, u(y)\, dy \,.
\end{align*}
This shows that
\begin{equation} \label{eq:heat}
k(x,y,t) = \sum_{l,v,\sigma} \sqrt{g_l(|x|)\ g_l(|y|)}\, \,
Y_{l,v,\sigma}(x)\, \overline{Y_{l,v,\sigma}(y)}\, \,
k_l(|x|,|y|,t) \,.
\end{equation}
In particular, on the diagonal we have
\begin{equation}
\label{eq:heatdiag} k(x,x,t) = k_0(|x|,|x|,t) + \sum_{l=1}^\infty \frac{b_l-1}{b_l}\, \,
k_l(|x|,|x|,t) \,.
\end{equation}
Note that the sum involves only finitely many terms since
the $l$'th summand is non-zero only when $|x|<r_l$. Here is a
relation between the heat kernels for different values of $l$:

\begin{lemma}\label{comp}
For all $l\geq 1$ one has
\begin{equation} \label{link}
k_l(r,s,t) \leq b_0 \cdots b_l\ k_0(r,s,t).
\end{equation}
\end{lemma}

\begin{proof}
Let $\chi_l$ be the characteristic function of the interval $(0,r_l)$ and
consider the family
of operators $B_l^M = A_0 +M\chi_l$ in $L_2(\R_+,g_0dr)$ for a constant $M> 0$.
With Trotter's product formula one sees that
$$
\exp(-tB_l^M)(r,s) \leq k_0(r,s,t) \,.
$$
Moreover, by monotone convergence $B_l^M$ converges as $M\to\infty$
in strong resolvent sense to the operator $0\oplus \tilde A_l$ in
$L_2((0,r_l),g_0dr)\oplus L_2((r_l,\infty),g_0dr)$, where $\tilde A_l$
is associated to the quadratic form $\int_{r_l}^\infty |u'|^2 g_0
\,dr$ with boundary condition $u(r_l)=0$. Hence for $r,s\geq r_l$,
\begin{equation}
\exp(-t\tilde A_l)(r,s) \leq k_0(r,s,t) \,.
\end{equation}
The operator $U$ of multiplication by the constant $\sqrt{b_0\cdots
b_l}$ maps $L_2((r_l,\infty),g_0dr)$ unitarily onto
$L_2((r_l,\infty),g_ldr)$ and satisfies $U^* A_l U= \tilde A_l$. Hence
$U^* \exp(-t A_l) U= \exp(-t \tilde A_l)$ which means
$$
k_l(r,s,t) = b_0\cdots b_l \exp(-t\tilde A_l)(r,s) \,.
$$
This proves inequality \eqref{link}.
\end{proof}

\begin{theorem} \label{main-estim}
Let $\Gamma$ be a symmetric tree. Then for any $x\in\Gamma$ we have
\begin{align}
k(x,x,t)\, & \leq \, k_0(|x|, |x|, t)\, g_0(|x|), \label{D}
\end{align}
\end{theorem}

\begin{proof}
Let $x\in\Gamma$ with $r_L<|x|\leq r_{L+1}$. Combining equation \eqref{eq:heatdiag} and Lemma \ref{comp} we obtain
\begin{align*}
k(x,x,t) & = k_0(|x|,|x|,t) + \sum_{l=1}^L \frac{b_l-1}{b_l} k_l(|x|,|x|,t)
 \leq k_0(|x|,|x|,t) \left( 1+ \sum_{l=1}^L b_0 \cdots b_{l-1} (b_l-1) \right) \\
 & \\
& = k_0(|x|,|x|,t)\, b_0 \cdots b_L = k_0(|x|,|x|,t)\, g_0(|x|) \,,
\end{align*}
as claimed.
\end{proof}

\smallskip


\section{Proof of the main results}
 \label{sect:proofs}

In this section we shall prove Theorems \ref{hk-estimate}, \ref{no-vd} and \ref{infinite-dim}. According to Theorem \ref{main-estim} the proof of upper bounds on the heat kernel of $-\Delta_\Gamma^N$ on the tree $\Gamma$ is reduced to the proof of upper bounds on the heat kernel of the operator $A_0$ on the half-line. To obtain such bounds for $A_0$ under various assumptions on $g_0$ is the topic of the following three subsections.


\subsection{Proof of Theorem \ref{hk-estimate}}

In order to derive our first bound on the heat kernel of $A_0$ we make use of a powerful equivalence theorem due to Grigor'yan and Saloff-Coste; see the surveys \cite{grig,sc} and the references therein. According to this result, a heat kernel bound follows from a volume doubling condition and a family of Poincar\'e inequalities for the weighted manifold $M(g_0)=([0,\infty), g_0\, dr)$. These bounds will be shown below in Lemmas~\ref{vd} and \ref{pi}, respectively.
 
Let $B(z,r)$ be the ball of radius $r$ centered in $z$ and let $V(z,r)$ be its volume in $M(g_0)$. More explicitly,
$$
B(z,r) = (\max\{z-r,0\},z+r)\,,
\qquad
V(z,r) = \int _{B(z,r)} g_0(s)\,ds \,.
$$ 
We also recall the doubling condition \eqref{xdouble} on $g_0$.

\smallskip

\begin{lemma}[Volume doubling] \label{vd}
Assume \eqref{xdouble}. Then the weighted manifold 
$M(g_0)$ satisfies the following volume doubling condition: 
for any ball $B(z,r)$  we have 
\begin{equation} \label{eq:vd}
V(z,r) \, \leq\,  2\, C_0 \, V(z, r/2).
\end{equation}
\end{lemma}

\begin{proof}
Obviously, $\frac{d}{dr} V(z,r) = g_0(z+r)+g_0(z-r)$ (with the convention that $g_0(t)=0$ if $t<0$). Moreover, by \eqref{xdouble} and by the fact that $g_0$ is non-decreasing,
$$
g_0(z+r)+g_0(z-r) \leq C_0 \left( g_0(\frac{z+r}2) + g_0(\frac{z-r}2) \right)
\leq C_0 \left( g_0(z+\frac{r}2) + g_0(z-\frac{r}2) \right) \,.
$$
This means that for any fixed $z$, the function $w(r)=V(z,r)$ satisfies $w'(r) \leq C_0 w'(r/2)$. Thus
$$
w(r) = \int_0^r w'(s)\,ds \leq C_0 \int_0^r w'(s/2)\,ds = 2 C_0 \int_0^{r/2} w'(t)\,dt = 2 C_0\, w(r/2) \,,
$$
as claimed.
\end{proof}

\begin{lemma}[Poincar\'e inequality] \label{pi}
The weighted manifold  $M(g_0)$ satisfies the following Poincar\' e inequality:  for any ball $B(z,r)$ and any $f\in H^1(B(z,r))$ it holds
\begin{equation} \label{eq:pi}
\inf_{\xi\in\R}\, \int_{B(z,r)}  |f(s)-\xi|^2\, g_0(s) \, ds \, \leq\, 4\, r^2 \int_{B(z,r)}
|f'(s)|^2\, g_0(s)\, ds \,.
\end{equation}
\end{lemma}

Of course, the infimum on the left side is attained at $\xi=V(z,r)^{-1}\int_{B(z,r)} f g_0 \,ds$. We emphasize that assumption \eqref{xdouble} is not needed in this lemma.

\begin{proof}
The key ingredient in the proof is the inequality
$$
\int_a^b |w|^2 |u|^2 \,ds \leq 4 \sup_{a\leq s\leq b}\left( \int_a^s |w|^2 \,dt \int_s^b \frac{dt}{|v|^2} \right) \int_a^b |v|^2 |u'|^2 \,ds
$$
for all $u$ with $u(b)=0$, which is an easy consequence of \cite[Thm.6.2]{OK}. Given a function $f\in H^1(B(z,r))$ we apply this bound to $u(s)=f(s)-f(z+r)$ and learn that
$$
\inf_{\xi\in\R}\, \int_{B(z,r)}  |f-\xi|^2\, g_0 \, ds 
\leq \int_{B(z,r)}  |u|^2\, g_0 \, ds \leq 4C \int_{B(z,r)} |u'|^2 g_0\,ds = 4C \int_{B(z,r)} |f'|^2 g_0\,ds
$$
with constant
$$
C= \sup_{(z-r)_+ \leq s\leq z+r}\left( \int_{(z-r)_+}^s g_0(t) \,dt \int_s^{z+r} \frac{dt}{g_0(t)} \right) \,.
$$
The monotonicity of $g_0$ implies that
$$
C \leq \sup_{(z-r)_+ \leq s\leq z+r}\left( \left(s-(z-r)_+ \right) g_0(s) \left(z+r-s\right) g_0(s)^{-1} \right) \leq r^2 \,,
$$
which proves the claimed bound.
\end{proof}

\begin{proposition}[One-dimensional heat kernel bound I] \label{basic-estim}
Under condition \eqref{xdouble} the heat kernel $k_0(r,s,t)$ of $A_0$ satisfies 
\begin{equation} \label{k0-bothsided}
\frac{c^{-1}}{\sqrt{t}\, \, g_0(r+\sqrt{t})}\, \, \leq\,  k_0(r,r,t)
\, \, \leq \, \frac{c}{\sqrt{t}\, \, g_0(r+\sqrt{t})}
\end{equation}
for some $c>0$, all $r>0$ and all $t>0$.
\end{proposition}

\begin{proof}
By Lemmas \ref{vd} and \ref{pi} the weighted manifold $M(g_0)$ satisfies the volume doubling condition and the Poincar\'e inequality. The Grigor'yan--Saloff-Coste theorem (see, e.g., \cite[Thms. 6.1 and 6.2]{grig} or \cite[Thm. 3.1]{sc}) then implies that
\begin{equation} \label{ball}
\frac{c^{-1}}{V(r,\sqrt{t})}\, \, \leq\,  k_0(r,r,t) \, \, \leq \,
\frac{c}{V(r,\sqrt{t})}\, .
\end{equation}
Clearly, we have  
$$
V(r,\sqrt{t}) \leq 2\sqrt{t}\ g_0(r+\sqrt{t}). 
$$
On the other hand, from \eqref{xdouble} it easily follows that 
\begin{align*}
V(r,\sqrt{t}) & \geq \int_{r+\sqrt{t}/2}^{r+\sqrt{t}} g_0(s)\, ds
\geq  C_0^{-1}\, \int_{r+\sqrt{t}/2}^{r+\sqrt{t}} g_0(2s)\, ds 
= (2 C_0)^{-1}\, \int_{2r+\sqrt{t}}^{2r+2\sqrt{t}} g_0(s)\, ds \\
& \geq (2C_0)^{-1}\, \sqrt{t}\, \, g_0(r+\sqrt{t}).
\end{align*}
In view of \eqref{ball} this proves the statement. 
\end{proof}

\begin{proof}[Proof of Theorem \ref{hk-estimate}]
The statement is an immediate consequence of equation \eqref{eq:heatdiag}, Theorem \ref{main-estim} and Proposition \ref{basic-estim}.
\end{proof}


\subsection{Proof of Theorem \ref{no-vd}}

We now turn to metric trees which do not necessarily satisfy the volume doubling property. Instead, we assume that $\int_0^\infty g_0(s)^{-1} \,ds$ is finite and that the quantity $S_\Gamma(\delta)$ in \eqref{eq:sgamma} is positive for some $\delta>2$. We also recall that $\tilde S_\Gamma(\delta)$ is defined in \eqref{eq:sgammatilde}. We shall deduce the 

\begin{proposition}[Nash inequality]\label{nashtrans}
Assume that $g_0$ satisfies \eqref{eq:sgamma} for some $\delta>2$. Then for all $f\in
H^1(\R_+,g_0)\cap L_1(\R_+,g_0)$ it holds 
\begin{equation}
\left( \int_0^\infty |f'|^2 g_0 \,dr \right)^{1/2}
\left(\int_0^\infty |f| g_0 \,dr\right)^{2/\delta} \geq \tilde
S_\Gamma^{1/2}(\delta) \left( \int_0^\infty |f|^2 g_0 \,dr
\right)^{(\delta+2)/2\delta} \,.
\end{equation}
\end{proposition}

\begin{proof}
By the one-dimensional Hardy--Sobolev inequality \cite[Thm.6.2]{OK},
$$
\int_0^\infty |f'|^2 g_0 \,dr \geq \tilde S_\Gamma(\delta) \left(
\int_0^\infty |f|^q g_0 \,dr \right)^{2/q}
$$
where $q=2\delta/(\delta-2)$. By H\"older's inequality, we have for any $0<\theta< 1$,
$$
\int_0^\infty |f|^2 g_0 \,dr \leq \left(\int_0^\infty |f| g_0
\,dr\right)^\theta \left(\int_0^\infty |f|^{(2-\theta)/(1-\theta)}
g_0 \,dr\right)^{1-\theta} \,.
$$
Choosing $\theta=(q-2)/(q-1)=4/(\delta+2)$ we obtain the assertion.
\end{proof}

\begin{proposition}[One-dimensional heat kernel bound II] \label{A0}
Under condition \eqref{eq:sgamma} for some $\delta>2$ the heat kernel $k_0(r,s,t)$ of $A_0$ satisfies 
$$
\sup_{r>0}\, k_0(r,r,t)
\leq \Big(\frac \delta{2\ \tilde S_\Gamma(\delta)}\Big)^{\delta/2} t^{-\delta/2}\qquad t>0.
$$
\end{proposition}

\begin{proof}
This is a consequence of Proposition \ref{nashtrans} and the classical Nash argument (see, e.g., \cite[Thm.
8.16]{LL}). The fact that $\exp(-tA_0)$ is positivity preserving and sub-Markovian follows from the Beurling--Deny theorem \cite[Sec. XIII.12]{ReSi4}.
\end{proof}

\begin{remark}
Conversely, if there exists an $S>0$ and a $\delta>2$ such that
$\exp(-tA_0)(r,r) \leq (S t)^{-\delta/2}$, then
$$
\sup_{r>0} \left( \int_0^r g_0(s)\,ds \right)^{(\delta-2)/\delta} \left(
\int_r^\infty \frac{ds}{g_0(s)} \right) \leq C_\delta\,  S^{-1}
$$
for some constant $C_\delta$ depending only on $\delta$ (and not on $g_0$).
Indeed, the heat kernel estimate implies a Nash inequality with
equivalent constants \cite[Thm. 8.16]{LL}, the Nash inequality
implies a Sobolev inequality with equivalent constants \cite{BCLS}
and the Sobolev inequality implies the above condition by the
characterization of the one-dimensional Hardy--Sobolev inequality.
\end{remark}

\begin{proof}[Proof of Theorem \ref{no-vd}]
The statement follows from Theorem \ref{main-estim} and Proposition \ref{A0}. 
\end{proof}


\subsection{Proof of Theorem \ref{infinite-dim}}

Finally, we discuss heat kernel bounds of homogeneous trees. In order to emphasize the dependence on the branching number $b\geq 2$ we write $A_b:=A_0$ for the non-negative operator in $L_2(\R_+,g_b)$ corresponding to the quadratic form
$$
\int_0^\infty |f'|^2 g_{b}\, dr , \qquad f\in H^1(\R_+, g_{b}).
$$
Moreover, let $\omega_b$ be the function on $\R_+$ satisfying
\begin{equation*}
-(g_b \, \omega_b')' = \lambda_b\, g_b\, \omega_b
\quad \mathrm{in}\ \R_+\setminus\N \,,
\end{equation*}
\begin{equation*}
\omega_b'(0)=0, \quad \omega_b(j+)= \omega_b(j-),\quad
\omega_b'(j-)=b\, \omega_b'(j+), \quad j\in\N\, .
\end{equation*}
We know from \cite[Lemma 4.2]{efk} that there is a positive constant $c_b $ such that
\begin{equation} \label{two-sided}
c_b^{-1}\, \frac{1+r}{\sqrt{g_{b}(r)}} \,  \leq \,  \omega_b(r)\,  \leq \, c_b \ 
\frac{1+r}{\sqrt{g_{b}(r)}}.
\end{equation}
Hence
$$
S_b^{-1} := \sup_{r>0} \left( \int_0^r \omega_b(s)^2 g_{b}(s)\,ds
\right)^{1/3} \left( \int_r^\infty \frac{ds}{\omega_b(s)^2 g_{b}(s)}
\right) <\infty \,.
$$
We write
$$
\tilde S_b := (3/ 4)^{4/3}\, S_b.
$$

\begin{proposition}[One-dimensional heat kernel bound III] \label{homog}
For every $b\geq 2$ the heat kernel $k_b(r,s,t)$ of $A_b$ satisfies
$$
k_b(r,r,t)  \, \leq \, \left(\frac 3{2\ \tilde S_b}\right)^{3/2}\, 
e^{-t\lambda_b}\,  t^{-3/2} \,  \omega^2_b(r).
$$
\end{proposition}

The strategy of the following proof is similar in spirit to an argument in \cite{da97}.

\begin{proof}
We consider the unitary operator $U: L_2(\R_+,g_b)\to
L_2(\R_+,\omega_b^2 g_b)$ of multiplication by $\omega_b^{-1}$. The
ground state representation
$$
\int_0^\infty |f'|^2\, g_b \,dr - \lambda_b \int_0^\infty |f|^2 \,
g_b \,dr = \int_0^\infty |h'|^2\, \omega_b^2 \, g_b \,dr
$$
with $f=\omega_b h = U^* h$ implies that $A_b-\lambda_b = U^* B_b
U$, where $B_b$ is the non-negative operator in $L_2(\R_+,\omega_b^2
g_b)$ corresponding to the quadratic form
$$
\int_0^\infty |h'|^2 \, \omega_b^2\, g_{b} \, dr
$$
with form domain $H^1(\R_+,\omega_b^2 g_b)$. This implies that
$\exp(-tA_b) = e^{-t\lambda_b} U^* \exp(-tB_b) U$ which, in terms of
the integral kernels, reads
$$
\exp(-tA_b)(r,s) = e^{-t\lambda_b} \omega_b(r) \exp(-tB_b)(r,s)
\omega_b(s) \,.
$$
Hence we need to prove that
\begin{equation} \label{needed}
\exp(-tB_b)(r,s) \leq \left(\frac 3{2\ S_b}\right)^{3/2} t^{-3/2}
\,.
\end{equation}
From \eqref{two-sided}, the definition of $\tilde S_b$, and the
one-dimensional Hardy--Sobolev inequality, see \cite[Thm.6.2]{OK}, we obtain
$$
\int_0^\infty |f'|^2\, \omega_b^2\, g_{b} \,dr \geq \tilde S_\Gamma
\left( \int_0^\infty |f|^6 \omega_b^2\, g_{b} \,dr \right)^{1/3}
$$
By H\"older,
$$
\int_0^\infty |f|^2 \omega_b^2\, g_{b}  \,dr \leq
\left(\int_0^\infty |f| \omega_b^2\, g_{b} \,dr \right)^{4/5}
\left(\int_0^\infty |f|^6 \omega_b^2\, g_{b}
\, dr\right)^{1/5} \,.
$$
Hence
\begin{equation}
\left( \int_0^\infty |f'|^2 \omega_b^2\, g_{b} \,dr \right)^{1/2}
\left(\int_0^\infty |f| \omega_b^2\, g_{b} \,dr\right)^{2/3} \geq
\tilde S_b^{1/2} \left( \int_0^\infty |f|^2 \omega_b^2\, g_{b} \,dr
\right)^{5/6} \,.
\end{equation}
Estimate \eqref{needed} then follows again from Nash's argument (see, e.g., \cite[Thm. 8.16]{LL}).
\end{proof}


\begin{proof}[Proof of Theorem \ref{infinite-dim}]
The statement follows from Theorem \ref{main-estim}, Proposition \ref{homog} and equation \eqref{two-sided}. 
\end{proof}

\section{Applications: spectral estimates for Schr\"odinger operators}
\label{sect:appl}

Consider a Schr\"odinger operator $-\Delta^N_\Gamma -V$ in $L_2(\Gamma)$
with (minus) an electric potential $V:\Gamma\to [0,\infty)$ decaying at infinity. One of the classical problems of spectral theory is to estimate moments of negative eigenvalues of $-\Delta^N_\Gamma -V$ in terms of an $L_p-$norm of $V$.  Estimates of the form 
\begin{equation} \label{eq:efk}
\tr\left(-\Delta^N_\Gamma -V\right)_-^\gamma \, \leq \,  C \int_{\Gamma} V(x)^{\gamma+\frac{a+1}{2}}\,
g_0(|x|)^{\frac{a}{d-1}}\, dx, \qquad \gamma\geq 0
\end{equation}
were proved in \cite{efk} under the assumption that $V(x)=V(|x|)$. The allowed values of $\gamma$ and $a$ in \eqref{eq:efk} are determined by the global dimension of $\Gamma$, and the constant $C$ depends on $a$ and $\gamma$ but not on $V$.  For $\gamma=0$ the quantity 
$$
\tr\left(-\Delta^N_\Gamma -V\right)_-^0 = N(-\Delta^N_\Gamma -V)
$$ 
coincides with the number of negative eigenvalues of $-\Delta^N_\Gamma -V$ (counted with their multiplicities).  The heat kernel estimates proven in the previous section allow us to extend some of the results obtained in \cite{efk} also to non-symmetric potentials.  Our approach is based on a well-known inequality due to Lieb \cite{lieb} which, in combination with the identity
\begin{equation}
\label{eq:riesz}
\tr\left(-\Delta^N_\Gamma -V\right)_-^\gamma = \gamma\, \int_0^\infty\,
\tau^{\gamma-1}\,
N\left(-\Delta^N_\Gamma -V+\tau\right)\, d\tau \,,
\end{equation}
yields
\begin{equation} \label{lieb}
\tr\left(-\Delta^N_\Gamma -V\right)_-^\gamma \, \leq \,
M_{\beta,\gamma}\, \int_\Gamma\, \int_{0}^\infty  k(x,x,t)\,
t^{-1-\gamma}\, (t\, V(x)-\beta)_+\, dt\, dx,
\end{equation}
where $\beta>0$ is arbitrary and
\begin{equation} \label{constant}
M_{\beta,\gamma}= \Gamma(\gamma+1)\, \left(
e^{-\beta}-\beta\, \int_\beta^\infty\, s^{-1}\, e^{-s}\,
ds\right)^{-1}.
\end{equation}
This inequality reduces bounds on $\tr\left(-\Delta^N_\Gamma -V\right)_-^\gamma$ to bounds on the heat kernel $k(x,x,t)$.

From \cite{efk} we recall the bound
\begin{equation}\label{eq:halflarge}
\tr\left(-\Delta^N_\Gamma -V\right)_-^{1/2} \, \leq \,
\int_{\Gamma} V(x) \, dx
\end{equation}
with sharp constant. By the Aizenman--Lieb argument, this also implies that
\begin{equation*}
\tr\left(-\Delta^N_\Gamma -V\right)_-^\gamma \, \leq \, 4\, \frac{\Gamma(\gamma+1)}{(4\pi)^{1/2} \Gamma(\gamma+3/2)} \int_{\Gamma} V(x)^{\gamma+1/2} \, dx
\end{equation*}
for $\gamma\geq 1/2$. When $\Gamma$ has global dimension one, these bounds are essentially (possibly up to the factor $4$ in front of the quotient) best possible. In the following we shall assume that $\Gamma$ has global dimension larger than one.


\subsection{Two-term estimates}
As a first application of our heat kernel bounds we state a two-term upper bound on $\tr\left(-\Delta^N_\Gamma -V\right)_-^\gamma$ for symmetric trees with finite global dimension $d>1$. Given a potential $V$ and a parameter $\beta>0$, we partition $\Gamma$ as follows,
\begin{equation} \label{partition}
\Gamma_\beta^- = \big\{x\in \Gamma\, : \, V(x)\, g_0(|x|)^{\frac{2}{d-1}} < \beta \big\}, \qquad \Gamma_\beta^+ = \Gamma\setminus\Gamma_\beta^- \,.
\end{equation}

\begin{theorem} \label{two-term1}
Let $\Gamma$ be a symmetric tree with global dimension $d> 1$ and let $\beta>0$.
\begin{enumerate}
\item Assume that $\gamma > 1/2$. Then
\begin{equation} \label{gamma>1/2}
\tr\left(-\Delta^N_\Gamma -V\right)_-^\gamma \, \leq \, L_d(\beta,\gamma)
\int_{\Gamma_\beta^-} V(x)^{\gamma+\frac{d}{2}}\, g_0(|x|)\, dx +
\tilde L_d(\beta, \gamma) \int_{\Gamma_\beta^+} V(x)^{\gamma+\frac{1}{2}}\, dx.
\end{equation}

\item Assume that $1-d/2<\gamma\leq 1/2$ if $1<d\leq 2$ or that $0\leq\gamma\leq 1/2$ if $d>2$. Then
\begin{equation} \label{gamma<1/2}
\tr\left(-\Delta^N_\Gamma -V\right)_-^\gamma \, \leq \, L_d(\beta,\gamma)
\int_{\Gamma_\beta^-} V(x)^{\gamma+\frac{d}{2}}\, g_0(|x|)\, dx +
\tilde L_d(\beta, \gamma) \int_{\Gamma_\beta^+} V(x) \, g_0(|x|)^{\frac{1-2\gamma}{d-1}}\, dx.
\end{equation}
\end{enumerate}

\noindent The constants in the above estimates are given by 
\begin{align*}
L_d(\beta, \gamma) & =\frac{C\, M_{\beta,\gamma}\, \beta^{1-\frac{d}{2}-\gamma}}{(\gamma+\frac d2-1)(\gamma+\frac d2)} \quad \text{if \ } \gamma\neq \frac 12, \qquad L_d(\beta, 1/2) = 2^{\frac{d+5}{2}}\, M_{\beta,\frac 12}\, \frac{C\, \beta^{\frac{1-d}{2}}}{d^2-1}   \\
& \\
\tilde L_d(\beta, \gamma) & =  M_{\beta,\gamma}\, \beta^{\frac{1}{2}-\gamma}  \Big( \frac{\pi^{-1/2}}{ |\gamma-\frac 12|}  +\frac{C}{\gamma+\frac d2-1}\Big) \quad \text{if \ } \gamma\neq \frac 12, \qquad \tilde L_d(\beta, 1/2) = 2, 
\end{align*}
where $C$ is the constant from Corollary \ref{finite-dim}.
\end{theorem}

The two terms on the right side of inequality \eqref{gamma>1/2} reflect the behavior of $\tr\left(-\Delta^N_\Gamma -V\right)_-^\gamma$ in the weak and strong coupling regime, i.e., for $V\to 0$ and $V\to\infty$, respectively. As it was shown in \cite{efk} these limiting behaviors are determined by the local and global dimension of $\Gamma$ which are equal to $1$ and $d$, respectively. Similar two-term estimates for Schr\"odinger operators on manifolds with different local and global dimensions were discussed recently in \cite{efk,rs,MoVa} and the references therein.

\begin{proof}
We split the double integral on the right side of \eqref{lieb} into four integral $I^{\sigma,\tau}$ with $\sigma,\tau\in\{\pm\}$. In the notation $I^{\sigma,\tau}$ the first index $\sigma$ indicates that the $x$ integration is over $\Gamma_\beta^\sigma$. The second index $\tau=-$ means that the $t$-integration is restricted to $(0, g_0(|x|)^{2/(d-1)})$, and correspondingly for $\tau=+$.

In the integrals $I^{\sigma,-}$ with $\tau=-$ we use the bound $k(x,x,t)\leq (\pi t)^{-1/2}$ from Theorem \ref{1-dim}. Since $t V(x) -\beta\leq 0$ for all $x\in \Gamma_\beta^-$ and $t\in (0, g_0(|x|)^{2/(d-1)})$, we have $I^{--}=0$. Moreover,
\begin{align*}
I^{+-} 
& \leq \pi^{-\frac 12} \int_{\Gamma_\beta^+}\, \int_{0}^{g_0(|x|)^{2/(d-1)}} t^{-\frac32-\gamma}\, (t\, V(x)-\beta)_+\, dt\, dx \\
& = \pi^{-\frac 12} \beta^{-\gamma+\frac12}
\int_{\Gamma_\beta^+} V(x)^{\gamma+\frac12} \ I_1(\beta^{-1} V(x) g_0(|x|)^\frac{2}{d-1}) \,dx \,,
\end{align*}
where
$$
I_1(a) = \int_{1}^{a} s^{-\frac32-\gamma}\, (s-1) \,ds \,,
\quad a\geq 1 \,.
$$
It is elementary to estimate
$$
I_1(a) \leq
\left\{
\begin{array}{l@{\qquad}l}
 (\gamma-\tfrac12)^{-1} & \mathrm{if}\ \gamma>\frac12 \, , \\
  &  \\
 (\tfrac12-\gamma)^{-1}\ a^{\tfrac12-\gamma} &  \mathrm{if}\ 0\leq \gamma<\frac12 \, .
 \end{array}
\right.
$$
This leads to the desired bound on $I^{+-}$. (For $\gamma=1/2$ we proceed differently below.)

We proceed to $I^{\sigma,+}$ with $\tau=+$ and recall the bound $k(x,x,t)\leq C t^{-d/2} g_0(|x|)$ from Corollary \ref{finite-dim}. We obtain
\begin{align*}
I^{-+} 
& \leq C \int_{\Gamma_\beta^-}\, g_0(|x|) \int_{g_0(|x|)^{2/(d-1)}}^\infty t^{-1-\frac d2-\gamma}\, (t\, V(x)-\beta)_+\, dt\, dx \\
& = C \beta^{-\gamma-\frac d2+1}
\left( \int_{\Gamma_\beta^-} V(x)^{\gamma+\frac d2}  g_0(|x|) \,dx \right)
\left( \int_1^\infty s^{-1-\frac d2-\gamma}\, (s-1) \,ds \right) \,.
\end{align*}
The latter integral is finite since $\gamma>1-d/2$. Finally, we have
\begin{align*}
I^{++} 
& \leq C \int_{\Gamma_\beta^+}\, g_0(|x|) \int_{g_0(|x|)^{2/(d-1)}}^\infty t^{-1-\frac d2-\gamma}\, (t\, V(x)-\beta)_+\, dt\, dx \\
& = C \beta^{-\gamma+\frac 12}
\int_{\Gamma_\beta^+}\,  V(x)^{\gamma+\frac 12}
\ I_2(\beta^{-1} V(x) g_0(|x|)^\frac{2}{d-1}) \,dx \,,
\end{align*}
where
$$
I_2(a) = a^\frac{d-1}2 \int_{a}^\infty s^{-1-\frac d2-\gamma}\, (s-1)\,ds \,,
\quad a\geq 1 \,.
$$
It is elementary to estimate
$$
I_2(a) \leq
\left\{
\begin{array}{l@{\qquad}l}
 (\gamma+\tfrac d2-1)^{-1} & \mathrm{if}\ \gamma>\frac12 \, , \\
  &  \\
 (\gamma+\tfrac d2-1)^{-1}\ a^{\tfrac12-\gamma} &  \mathrm{if}\ 0\leq \gamma<\frac12 \, .
 \end{array}
\right.
$$
In view of \eqref{lieb} we have
$$
\tr(-\Delta_\Gamma^N - V)^\gamma_- \, \leq\, M_{\beta,\gamma} \big (I^{++} + I^{+-}+I^{-+} \big).
$$
We have thus proven the statement of the Theorem  for $\gamma>1/2$ and for $1-d/2<\gamma<1/2$. 

For $\gamma=1/2$ we proceed in a different way. First, we note that by substituting the bound for $k(x,x,t)$ from Corollary \ref{finite-dim} into \eqref{lieb} we obtain 
\begin{equation}\label{eq:halfsmall}
\tr\left(-\Delta^N_\Gamma -V\right)_-^{1/2} \, \leq \, 4 M_{\beta,\frac 12}\, \frac{C\, \beta^{\frac{1-d}{2}}}{d^2-1}  \int_{\Gamma} V(x)^{\frac{1+d}{2}}\, g_0(|x|)\, dx \,.
\end{equation}

Next, we recall that for two lower semi-bounded, self-adjoint operators $H_1$ and $H_2$ one has $N(H_1+H_2)\leq N(H_1)+N(H_2)$ by the variational principle. (Here again, $N(\cdot)$ denotes the number of negative eigenvalues, counting multiplicities.) Applying this with $H_j$ replaced by $H_j+\tau/2$, i.e., $N(H_1+H_2+\tau) \leq N(H_1+\tau/2)+ N(H_2+\tau/2)$, and integrating with respect to $\tau$ we find in view of \eqref{eq:riesz}
$$
\tr(H_1+H_2)_-^\gamma =\gamma \int_0^\infty N(H_1+H_2+\tau) \tau^{\gamma-1}\,d\tau
\leq 2^\gamma \left( \tr(H_1)_-^\gamma + \tr(H_2)_-^\gamma \right) \,.
$$
Now for given potential $V$ and given parameter $\beta>0$, we decompose $V=V_<+V_>$ with $V=V_<$ on $\Gamma_\beta^-$ and $V=V_>$ on $\Gamma_\beta^+$. Applying the previous construction with $H_1= -\Delta_\Gamma^N - 2V_<$ and $H_2= -\Delta_\Gamma^N - 2V_>$ we conclude that
\begin{equation}\label{eq:split}
\tr(-\Delta_\Gamma^N - V)^\gamma_- \leq \tr( -\Delta_\Gamma^N - 2V_< )_-^\gamma + \tr(-\Delta_\Gamma^N - 2V_>)_-^\gamma \,.
\end{equation}
The claimed inequality now follows by applying \eqref{eq:halfsmall} and \eqref{eq:halflarge} to the first and second term, respectively.
\end{proof}

\begin{remark}
The operator $-\Delta^N_\Gamma -V$ has no weakly coupled eigenvalues if $d>2$. Indeed, recall the Hardy type inequality 
\begin{equation} \label{eq:hardy}
-\Delta^N_\Gamma \, \geq \, \frac{C}{1+|x|^2}, \qquad d>2,
\end{equation}
see \cite{efk1}. In view of \eqref{eq:hardy} it thus follows that for $\beta$ small enough the first term on the right hand side of \eqref{gamma>1/2} and \eqref{gamma<1/2} can be left out provided we modify the constant $\tilde L_d(\beta,\gamma)$. Indeed, using the notation of the proof of Theorem \ref{two-term1} we deduce from \eqref{eq:hardy} that for sufficiently small $\beta$ 
\begin{align*}
\tr(-\Delta_\Gamma^N - V)^\gamma_- & = \tr(-\Delta_\Gamma^N - V_< -V_>)^\gamma_- \leq \tr\big(-\frac 12 \Delta_\Gamma^N  +\frac{C}{2(1+|x|^2)} - V_< -V_> \big)^\gamma_- \\
& \quad \leq \tr\big(-\frac 12 \Delta_\Gamma^N  -V_> \big)^\gamma_- = 2^{-\gamma} \, \tr(- \Delta_\Gamma^N  -2 V_> )^\gamma_- \,,
\end{align*}
where we used that $V_< < \beta g_0(|x|)^{-2/(d-1)} \leq C(2(1+|x|^2))^{-1}$ if $\beta$ is sufficiently small.
\end{remark}


\subsection{One-term estimates}

The following statements are easy consequences of Theorem~\ref{two-term1}. We state them separately in order to show that they extend some of the results of \cite{efk} to general, not necessarily symmetric, potentials $V$.

\begin{corollary} \label{lt-extended}
Let $\Gamma$ be a symmetric tree with global dimension $d$.  Assume that either $1< d\leq 2$ and $0\leq a< d-1$, or else that $d> 2$ and $0\leq a \leq 1$. Then for any $\gamma\geq \frac{1-a}{2}$ there is a constant $K(a,d,\gamma)$ such  that 
\begin{equation} \label{eq:lt-ext}
\tr\left(-\Delta^N_\Gamma -V\right)_-^\gamma \, \leq \,
K(a,d,\gamma) \int_\Gamma V(x)^{\gamma+\frac{a+1}{2}}\, g_0(|x|)^{\frac{a}{d-1}}\, dx.
\end{equation}
\end{corollary}

The only case of \cite[Thm.2.7(1)]{efk}, which is not covered by this corollary, is $a=d-1$ for $1<d<2$.

\begin{proof}
Assume first that $\gamma=(1-a)/2$. Then $1-d/2<\gamma\leq 1/2$ if $1<d\leq 2$ and $0\leq\gamma\leq 1/2$ if $d>2$ and we are in the situation of the second item of Theorem \ref{two-term1}. For any $\beta>0$ we have $V(x)^{\gamma+\frac{d}{2}} \, g_0(|x|) \leq \beta^{\frac{d-a-1}{2}} V(x)^{\gamma+\frac{a+1}{2}}\, g_0(|x|)^{\frac{a}{d-1}}\, $ if $x\in\Gamma_\beta^-$ (since $d\geq a+1$) and $V(x) \, g_0(|x|)^{\frac{1-2\gamma}{d-1}} = V(x)^{\gamma+\frac{a+1}{2}}\, g_0(|x|)^{\frac{a}{d-1}}\, $ if $x\in\Gamma_\beta^+$.  From Theorem \ref{two-term1} we thus get inequality \eqref{eq:lt-ext} with the constant
$$
K(a,d, (1-a)/2) = \inf_{\beta>0}\, \max \Big \{ \beta^{\frac{d-a-1}{2}}\, L_d(\beta, (1-a)/2)\, ,\, \tilde L_d(\beta, (1-a)/2) \Big\}.
$$
For $\gamma> (1-a)/2$ the claim follows by the Aizenman--Lieb argument.
\end{proof}


\subsection{\bf Estimates for homogeneous trees}

As we have already mentioned, the infimum of the spectrum of the Laplace operator on a homogeneous tree with a branching number $b\geq 2$ is strictly positive and given by $\lambda_b$ from \eqref{lambdab}. It is therefore natural to look for estimates on $\tr\left(-\Delta^N_\Gamma -\lambda_b-V\right)_-^\gamma$. This is the content of the following

\begin{theorem} \label{lt-homog}
Let $\Gamma_b$ be a homogeneous tree with branching number $b\geq 2$. Then for any $\gamma\geq 0$ there is a constant $L_b(\beta,\gamma)$ such that
\begin{align*}
\tr\left(-\Delta^N_{\Gamma_b} -\lambda_b-V\right)_-^\gamma \ \leq \
L_b(\beta,\gamma) \int_{\Gamma_b} V(x)^{\gamma+\frac{3}{2}}\, (1+|x|)^2 \, dx \,.
\end{align*}
\end{theorem}

\begin{proof}
The analogue of \eqref{lieb} in this case is
$$
\tr\left(-\Delta^N_\Gamma -\lambda_b-V\right)_-^\gamma \, \leq \, \inf_{\beta>0}\, M_{\beta,\gamma} \int_\Gamma\, \int_{0}^\infty  \, k_b(x,x,t) \, e^{\lambda_b t}\, 
t^{-1-\gamma}\, (t\, V(x)-\beta)_+\, dt\, dx.
$$
The result then follows by using the estimate from Theorem \ref{infinite-dim}. 
\end{proof}

\medskip


\bibliographystyle{amsalpha}

\begin{thebibliography}{EKKST}

\bibitem[BCLS]{BCLS} D. Bakry, T. Coulhon, M. Ledoux, L. Saloff-Coste: Sobolev inequalities in disguise.
{\em Indiana Univ.~Math.~ J.}  \textbf{44} (1995), no. 4, 1033--1074.

\bibitem[BCFK]{BCFK} 
G. Berkolaiko, R. Carlson, St. A. Fulling, and P. Kuchment (eds.): Quantum graphs and their
applications, \textit{Contemporary Mathematics} \textbf{415}, Amer. Math. Soc., Providence, RI, 2006.

\bibitem[BF]{BF}
J. Breuer and R. L. Frank: Singular spectrum for radial trees. \textit{Rev. Math. Phys.} \textbf{21} (2009), no. 7, 1-17.

\bibitem[Ca]{Ca}
R. Carlson: Nonclassical Sturm-Liouville problems and Schr\"odinger operators on radial trees. \textit{Electron J. Differential Equation} \textbf{71} (2000), 24pp.

\bibitem[Cou]{cou} T.~Coulhon: Heat kernels and isoperimetry on non-compact Riemannian manifolds. P. Auscher, T. Coulhon, A. Grigor'yan (eds.), {\em Contemporary Mathematics}, Amer. Math. Soc. (2004) 65-99.

\bibitem[Da1]{Da} E.B.~Davies: {\em Heat kernels and spectral theory}, Cambridge University Press, Cambridge, 1989.

\bibitem[Da2]{da97}  E.B.~Davies: Non-Gaussian aspects of heat kernel behaviour, {\em 
J. London Math. Soc.} {\bf 55} (1997) 105--125.

\bibitem[DH]{dh}
S.~Demirel, E.~Harrell: On semiclassical and universal inequalities
for eigenvalues of quantum graphs, \emph{Rev. Math. Phys.} \textbf{22} (2010), no. 3, 305--329.

\bibitem[EFK1]{efk1} T.~Ekholm, R.L.~Frank and H.~Kova\v{r}\'{\i}k:
Remarks about Hardy inequalities on metric trees. P. Exner, et al.
(eds.) {\em Proc.~Sympos.~Pure.~Math.} {\bf 77} (2008) 369--379.

\bibitem[EFK2]{efk} T.~Ekholm, R.~L.~Frank, and H.~Kova\v r\'{\i}k:
Eigenvalue estimates for Schr\"odinger operators on metric trees.
{\em Adv.~in Math.} {\bf 226} (2011), no. 6, 5165--5197. 

\bibitem[EHP]{ehp} W.~D.~Evans, D.~J.~Harris and L.~Pick: Weighted Hardy
  and Poincar\'e inequalities on trees. {\em J.~London Math.~Soc.} (2)
  {\bf 52} (1995), no.1, 121--136.

 \bibitem[EKKST]{EKKST} P.~Exner, J.P.~Keating, P.~Kuchment, T.~Sunada, A.~Teplayaev (eds.): Analysis on graphs and its applications, Proc. Symp. Pure Math. \textbf{77}, Amer. Math. Soc., Providence, R.I., 2008.

\bibitem[FLS]{fls} R.L.~Frank, E.H.~Lieb and R.~Seiringer:
Equivalence of Sobolev inequalities and Lieb-Thirring inequalities. In: XVIth International Congress on Mathematical Physics, Proceedings of the ICMP held in Prague, August 3-8, 2009, P. Exner (ed.), 523-535, World Scientific, Singapore, 2010.

\bibitem[F]{F} L. Friedlander:  Extremal properties of eigenvalues for a metric graph.
{\em Ann.~Inst.~Fourier} \textbf{55} (2005), no. 1, 199--211.

\bibitem[Gr]{grig} A.~Grigor'yan: Heat kernels on weighted manifolds and applications, {\em Contemp. Mathematics} {\bf 398} (2006) 93--191.

\bibitem[KLW]{klw} M. Keller, D. Lenz and R. K. Wojciechowski: Volume growth, spectrum and stochastic 
completeness of infinite graphs. arXiv: 1105.0395v1. 

\bibitem[K]{k} H.~Kova\v r\'{\i}k: Weakly coupled Schr\"odinger operators
on regular metric trees, {\em SIAM J.~Math.~Anal.} {\bf 39} (2007) 1135-1149.  

\bibitem[LS]{LS} D. Levin and M. Solomyak: \textit{The Rozenblum–Lieb–Cwikel inequality for Markov generators}. J. Anal. Math. \textbf{71} (1997), 173--193.

\bibitem[L]{lieb} E.~Lieb: Bound states of the Laplace and Schr\"odinger operators,
{\em Bull.~Amer.~Math.~Soc.} {\bf 82} (1976) 751--753.

\bibitem[LL]{LL} E. H. Lieb and M. Loss: \textit{Analysis}.
Second edition. Graduate Studies in Mathematics \textbf{14},
American Mathematical Society, Providence, RI, 2001.

\bibitem[LT]{LT} E. H. Lieb and W. Thirring: \textit{Inequalities for the moments of the eigenvalues of the Schr\"odinger Hamiltonian and their relation to Sobolev inequalities}. Studies in Mathematical Physics, 269--303. Princeton University Press, Princeton, NJ, 1976.

\bibitem[Ma]{mz} V.~Maz'ya: {\em Sobolev Spaces},
Springer Verlag, Berlin New York, (1985).

\bibitem[MV]{MoVa} S. Molchanov and B. Vainberg: On general Cwikel-Lieb-Rozenblum and Lieb-Thirring inequalities, in: Around the research of Vladimir Maz'ya, A. Laptev (ed.), 201--246, International Mathematical Series 13 (2010).

\bibitem[NS1]{NS1}
K.~Naimark and M.~Solomyak: Geometry of the Sobolev spaces on the
symmetric trees and Hardy's inequalities,
 {\em Russian Journal of Math. Phys.} {\bf 8} (2001) 322--335.
 
 \bibitem[NS2]{NS2}
K.~Naimark and M.~Solomyak, Eigenvalue estimates for the weighted
Laplacian on metric trees, {\em Proc. London  Math. Soc.} {\bf 80}
(2000), no. 3, 690--724.

\bibitem[OK]{OK} B. Opic and A. Kufner, \textit{Hardy inequalities}, Pitman Research Notes in Mathematics \textbf{219}. Longman
Scientific \& Technical, Harlow, 1990.

\bibitem[RS4]{ReSi4}
M. Reed, B. Simon: \textit{Methods of modern mathematical physics. IV. Analysis of operators}. Academic Press, New York-London, 1978.

\bibitem[RS]{rs} G.~Rosenblum and M.~Solomyak: Counting Schr\"odinger boundstates:
semiclassics and beyond,  {\em Sobolev spaces in mathematics. II}, {\em  Int. Math. Ser. (N. Y.)} {\bf  9} 329--353.

\bibitem[Sa]{sc}
L. Saloff-Coste: The heat kernel and its estimates, \textit{Advanced Studies in Mathematics} (2009).

\bibitem[So]{sol}
M.~Solomyak, Remarks on counting negative eigenvalues of 
Schr\"odinger operators on symmetric metric trees, {\em Amer. Math. Soc. Transl.} {\bf 226} ser. 2, 
(2009) 165--178.

\bibitem[SS]{SS}
A.~Sobolev and M.~Solomyak: Schr\"odinger operators on homogeneous
metric trees: spectrum in gaps, {\em Rev. Math. Phys.} {\bf 14}
(2002) 421--467.



\end{thebibliography}

\end{document}